\documentclass[12pt]{amsart}
\usepackage{enumerate}
\usepackage{hyperref}
\usepackage{graphicx,color}
\usepackage{cite}
\usepackage{cancel}
\usepackage{amsthm, amsmath, amssymb, mathtools}
\usepackage[top=45truemm, bottom=45truemm, left=30truemm, right=30truemm]{geometry}

\renewcommand{\epsilon}{\varepsilon}

\newcommand{\abs}[1]{\left|#1\right|}

\numberwithin{equation}{section}

\newtheorem{theorem}{Theorem}[section]
\newtheorem{lemma}[theorem]{Lemma}

\theoremstyle{definition}

\newcommand{\cX}{\mathcal{X}}

\title[A simple proof of the formula of Solov'ev--Nielsen--Blom ...]{A simple proof of the formula of Solov'ev--Nielsen--Blom for the expected waiting time}

\author[Y.\ Yoshida]{Yuuya Yoshida}
\address{Yuuya Yoshida\\
Nagoya Institute of Technology\\ Gokiso-cho\\ Showa-ku\\ Nagoya\\ 466-8555\\ Japan}
\curraddr{}
\email{yyoshida9130@gmail.com}

 %Primary Secondary


\begin{document}
\maketitle

\begin{abstract}
Solov'ev (1966), Nielsen (1973), and Blom (1982) independently showed a formula for the expected waiting time until a given finite pattern first occurs in random data.
In this paper, we give a simple and combinatorial proof of the formula.
\end{abstract}

\section{Introduction}\label{intro}

Let $r\ge2$ be an integer, $\cX$ be the alphabet $\{0,1,\ldots,r-1\}$, and 
$p$ be a probability distribution on $\cX$ such that all $p(x)$ are positive.
We represent a finite pattern as a finite word over $\cX$.
Unless otherwise noted, finite words are considered over $\cX$.

Given a finite word $w$, 
we consider the expected waiting time $E(w)$ until $w$ first occurs in random data $\mathbf{X}=X_1X_2\cdots$\footnote{Since we interpret $X_1X_2\cdots$ as the product of words, it is not the product of integers.}, 
where $X_1,X_2,\ldots$ are independent and identically distributed random variables subject to $p$.
For the empty word $\emptyset$, we regard $E(\emptyset)$ as zero; 
see Section~\ref{EWT} for a more precise definition of $E(w)$.
Using the following theorem, one can compute $E(w)$ recursively.

\begin{theorem}[Solov'ev \cite{Solov'ev}, Nielsen \cite{Nielsen}, and Blom \cite{Blom}]\label{thm1} % Collings, 1977; 
	Let $w$ be a finite word of length $\ge1$, and $\hat{w}$ be the longest bifix of $w$.
	Then 
	\[
	E(w) = E(\hat{w}) + p(w)^{-1},
	\]
	where $p(w)$, $|w|$, and $w(i)$ denote the product $\prod_{i=1}^{|w|} p(w(i))$, the length of $w$, and the $i$th letter of $w$, respectively.
\end{theorem}

Here, a word $v\not=w$ is called a \textit{bifix} of a finite word $w$ 
if $w=vw_1=w_2v$ for some $w_1$ and some $w_2$.
For example, the empty word $\emptyset$ and the word $0$ are bifixes of the binary word $00110$.
Theorem~\ref{thm1} was first proved by Solov'ev \cite{Solov'ev}, and was rediscovered by Nielsen \cite{Nielsen} and Blom \cite{Blom}.

In this paper, we give a simple proof of Theorem~\ref{thm1}.
Although $E(w)$ is an object in probability theory, 
our proof is combinatorial rather than probabilistic, and is new to the best of our knowledge.
Furthermore, we see the following formulas as secondary results.
\begin{description}
	\item[F1]
	$\displaystyle \sum_{y\in\cX} p(y)E(wy) = E(w)+1+(r-1)p(w)^{-1}$ for every finite word $w$.\footnote{We define $p(\emptyset)=1$ for the empty word $\emptyset$.}
	\item[F2]
	$\displaystyle \sum_{|w|=n} p(w)E(w) = r^n+n-1$ for every integer $n\ge0$.
\end{description}
They are not strictly related to the proof of Theorem~\ref{thm1}, but 
we prove them in Section~\ref{formulas} because we do not know whether they are widely known.

\section{Expected waiting time}\label{EWT}

In this section, we define the waiting time $T_w$, 
the expected waiting time $E(w)$, and 
the conditional expected waiting time $E(w|w')$.
In particular, $E(w)$ is defined more precisely than the definition stated in Section~\ref{intro}.

First, we state several notations on words.
Denote by $w(i)$ (resp.\ $|w|$) the $i$th letter (resp.\ the length) of a word $w$.
For a word $w$ and integers $1\le i<j\le|w|$, 
denote by $w[i,j]$ the word $w(i)w(i+1)\cdots w(j)$.
If $w$ is a right-infinite word, denote by $w[i,\infty)$ the word $w(i)w(i+1)\cdots$ for an integer $i\ge1$.

\vskip1ex
\textbf{Definition of $T_w$.}
For $\xi=\xi(1)\xi(2)\cdots\in\cX^{\mathbb{N}}$, define $T_w(\xi)$ as 
$T_w(\xi)=\infty$ if $w$ never occurs in $\xi$, and 
\[
T_w(\xi) = \min\biggl\{ |w'w| : \xi=w'w\xi',\ w'\in\bigcup_{k=0}^\infty \cX^k,\ \xi'\in\cX^{\mathbb{N}} \biggr\}
\]
if $w$ occurs in $\xi$, where $\cX^0$ denotes the singleton $\{ \emptyset \}$.
That is, $T_w(\xi)$ is the time when $w$ first occurs in $\xi=\xi(1)\xi(2)\cdots$ 
if time $i$ denotes the $i$th position from the left.
Since every finite word $w$ occurs in random data $\mathbf{X}$ with probability $1$ (see Lemma~\ref{lem0} in Appendix), 
the value $T_w(\mathbf{X})$ is finite with probability $1$.

\vskip1ex
\textbf{Definition of $E(w)$.}
For a finite word $w$, define $E(w)$ as the expectation of $T_w(\mathbf{X})$.
It can easily be checked that $E(w)$ is always finite.
Indeed, it is clear that $E(\emptyset)=0$; 
if $w$ is a finite word of length $\ge1$, then  
\begin{align*}
	E(w) &= \sum_{q=1}^\infty \sum_{\abs{w}q\le k<\abs{w}(q+1)} k\mathbf{P}\{ T_w(\mathbf{X})=k \}
	\le \sum_{q=1}^\infty \sum_{\abs{w}q\le k<\abs{w}(q+1)} kp(w)(1-p(w))^{q-1}\\
	&< \sum_{q=1}^\infty \abs{w}^2(q+1)p(w)(1-p(w))^{q-1}
	< \infty,
\end{align*}
where $\mathbf{P}\{ T_w(\mathbf{X})=k \}$ is the probability that $T_w(\mathbf{X})=k$.

\vskip1ex
\textbf{Definition of $E(w|w')$.}
For finite words $w$ and $w'$, define $E(w|w')$ as the expectation of $T_w(w'\mathbf{X}) - |w'|$.
The value $E(w|w')$ can be regarded as the expected additional time until $w$ first occurs after $w'$ occurs.

\vskip1ex
Now, we state three lemmas.
The first two can be understood intuitively and easily, 
and can be proved along the definitions of $E(w)$ and $E(w|w')$.
The last one follows from the first two easily.

\begin{lemma}\label{lem1}
	If a word $w$ is the product $w_1w_2$ of finite words $w_1$ and $w_2$, then 
	\[
	E(w) = E(w|w_1) + E(w_1).
	\]
\end{lemma}

\begin{lemma}\label{lem2}
	For every finite word $w$ of length $n\ge1$, 
	\[
	E(w) = E(w[1,n-1]) + 1 + \sum_{y\not=w(n)} p(y)E(w|w[1,n-1]y),
	\]
	where $w[1,n-1]$ is interpreted as the empty word $\emptyset$ if $n=1$.
\end{lemma}

\begin{lemma}\label{lem3}
	Let $w$ be a finite word of length $n\ge1$.
	For $x\in\cX$, let $\hat{w}_x$ be the longest bifix of $w_x=w[1,n-1]x$, 
	where $w[1,n-1]$ is interpreted as the empty word $\emptyset$ if $n=1$.
	Then, for every $x\in\cX$, 
	\[
	p(x)E(w_x) = E(w[1,n-1]) + 1 - \sum_{y\not=x} p(y)E(\hat{w}_y).
	\]
\end{lemma}

For the proofs of Lemmas~\ref{lem1} and \ref{lem2}, see Appendix.
We prove only Lemma~\ref{lem3} in this section.

\begin{proof}[Proof of Lemma~$\ref{lem3}$ assuming Lemmas~$\ref{lem1}$ and $\ref{lem2}$]
	Let $x\in\cX$.
	By Lemmas~\ref{lem1}, 
	\[
	E(w_x) = E(w[1,n-1]) + 1 + \sum_{y\not=x} p(y)E(w_x|w_y).
	\]
	Since $E(w_x|w_y)=E(w_x|\hat{w}_y)$ for every $y\in\cX\setminus\{ x \}$, we have 
	\[
	E(w_x) = E(w[1,n-1]) + 1 + \sum_{y\not=x} p(y)E(w_x|\hat{w}_y).
	\]
	By this and Lemma~\ref{lem2}, 
	\[
	E(w_x) = E(w[1,n-1]) + 1 + (1-p(x))E(w_x) - \sum_{y\not=x} p(y)E(\hat{w}_y),
	\]
	which yields the desired equality.
\end{proof}

\section{Proof of Theorem~$\ref{thm1}$}

We prove Theorem~\ref{thm1} in this section.
The following lemma is the key to our proof.

\begin{lemma}\label{lem4}
	Let $w$ be a finite word of length $n\ge2$.
	For $y\in\cX$, let $\hat{w}_y$ be the longest bifix of $w_y=w[1,n-1]y$, and $n_y$ be the length of $\hat{w}_y$.
	Assume 
	\[
	n_x = \max_{y\in\cX} n_y.
	\]
	Then (i) $n_x\ge1$;
	(ii) $w[1,n_x-1]$ is the longest bifix of $w[1,n-1]$;
	(iii) $\hat{w}_y$ is the longest bifix of $w[1,n_x-1]y$ for every $y\in\cX\setminus\{x\}$.
\end{lemma}
\begin{proof}
	Proof of (i).
	It follows from $n\ge2$ that $n_x\ge n_{w(1)}\ge1$.
	\par
	Proof of (ii).
	Let $m$ be the length of the longest bifix of $w[1,n-1]$.
	Since $w[1,n_x-1]$ is a bifix of $w[1,n-1]$, the inequality $m\ge n_x-1$ holds.
	We show the opposite inequality.
	Since $w[1,m]$ is a bifix of $w[1,n-1]$, the word $w'=w[1,m+1]$ is a bifix of $w_{x'}$, where $x'=w(m+1)$.
	Thus, $m+1\le n_{x'}\le n_x$.
	Therefore, $m=n_x-1$, which implies (ii).
	\par
	Proof of (iii).
	By (i) and the definition of $n_y$, if $y\not=x$, then $n_y<n_x$.
	Thus, $\hat{w}_y$ is the longest bifix of $w[1,n_x-1]y=w[n-n_x+1,n-1]y$.
\end{proof}

Let us prove Theorem~\ref{thm1} by using Lemmas~\ref{lem3} and \ref{lem4}.

\begin{proof}[Proof of Theorem~$\ref{thm1}$]
	We show the assertion by induction on $|w|$.
	For $|w|=1$, the assertion follows from $E(\emptyset)=0$ and 
	$E(x) = \sum_{k=1}^\infty kp(x)(1-p(x))^{k-1} = p(x)^{-1}$.
	Fix $n\ge2$. Assume that the assertion holds for $|w|=n-1$.
	We show the assertion also holds for $|w|=n$.
	\par
	By the induction hypothesis and Lemma~\ref{lem4}, 
	\begin{equation}
		E(w[1,n-1]) = E(w[1,n_x-1]) + p(w[1,n-1])^{-1},
		\label{eq1}
	\end{equation}
	where $x$, $n_y$, $w_y$ and $\hat{w}_y$ are the same as in Lemma~\ref{lem4} for $y\in\cX$.
	By Lemma~\ref{lem3}, 
	\begin{equation}
		p(x)E(w_x) = E(w[1,n-1]) + 1 - \sum_{y\not=x} p(y)E(\hat{w}_y).
		\label{eq2}
	\end{equation}
	By Lemma~\ref{lem4}, $\hat{w}_y$ is the longest bifix of $w[1,n_x-1]y$ for every $y\in\cX\setminus\{x\}$.
	This and Lemma~\ref{lem3} imply that 
	\begin{equation}
		p(x)E(\hat{w}_x) = E(w[1,n_x-1]) + 1 - \sum_{y\not=x} p(y)E(\hat{w}_y).
		\label{eq3}
	\end{equation}
	By \eqref{eq2}, \eqref{eq3} and \eqref{eq1}, 
	\begin{align*}
		p(x)E(w_x) - p(x)E(\hat{w}_x)
		&= E(w[1,n-1]) - E(w[1,n_x-1])\\
		&= p(w[1,n-1])^{-1} = p(w_x)^{-1}p(x).
	\end{align*}
	Thus, 
	\begin{equation}
		E(w_x) - E(\hat{w}_x) = p(w_x)^{-1}.
		\label{eq4}
	\end{equation}
	By Lemma~\ref{lem3}, for every $z\in\cX$, 
	\begin{equation}
		p(z)E(w_z)-p(z)E(\hat{w}_z) = E(w[1,n-1])+1-\sum_{y\in\cX} p(y)E(\hat{w}_y).
		\label{eq1''}
	\end{equation}
	Substituting $z=w(n),x$ into \eqref{eq1''}, we obtain 
	\[
	p(w(n))E(w)-p(w(n))E(\hat{w}) = p(x)E(w_x)-p(x)E(\hat{w}_x).
	\]
	By this and \eqref{eq4}, the assertion also holds for $|w|=n$.
	Therefore, we complete the proof.
\end{proof}

\section{Proofs of \upshape{F1} and \upshape{F2}}\label{formulas}

Using Lemma~\ref{lem3} and Theorem~\ref{thm1}, we prove F1 and F2 in this section.

\begin{proof}[Proof of \upshape{F1}]
	Let $w$ be a finite word.
	By Lemma~\ref{lem3}, for every $x\in\cX$, 
	\begin{equation}
		p(x)E(wx) = E(w)+1 - \sum_{y\in\cX} p(y)E(\widehat{wy})+p(x)E(\widehat{wx}),
		\label{eq6}
	\end{equation}
	where $\widehat{wy}$ denotes the longest bifix of $wy$ for $y\in\cX$.
	Thus, 
	\begin{equation}
	\begin{split}
		\sum_{x\in\cX} p(x)E(wx)
		&= rE(w)+r-r\sum_{y\in\cX} p(y)E(\widehat{wy})+\sum_{x\in\cX} p(x)E(\widehat{wx})\\
		&= r(E(w)+1) - (r-1)\sum_{y\in\cX} p(y)E(\widehat{wy}).
	\end{split}\label{eq7}
	\end{equation}
	By \eqref{eq6} and Theorem~\ref{thm1}, 
	\begin{equation}
	\begin{split}
		\sum_{y\in\cX} p(y)E(\widehat{wy})
		&= E(w)+1+p(x)E(\widehat{wx})-p(x)E(wx)\\
		&= E(w)+1-p(w)^{-1}.
	\end{split}\label{eq5}
	\end{equation}
	Substituting \eqref{eq5} into \eqref{eq7}, we obtain F1.
\end{proof}

\begin{proof}[Proof of \upshape{F2}]
	For an integer $n\ge0$, set $S_n=\sum_{|w|=n} p(w)E(w)$.
	It is clear that $S_0=0$.
	By F1, for every integer $n\ge0$, 
	\begin{align*}
		S_{n+1} &= \sum_{|w|=n} p(w)\sum_{y\in\cX} p(y)E(wy)\\
		&= \sum_{|w|=n} p(w)\bigl( E(w)+1+(r-1)p(w)^{-1} \bigr)\\
		&= S_n+1+(r-1)r^n.
	\end{align*}
	Solving this recurrence relation with $S_0=0$, 
	we obtain $S_n=r^n+n-1$.
\end{proof}

\section*{Acknowledgments}
The author was supported by JSPS KAKENHI Grant Numbers JP22J00339 and JP22KJ1621.

\appendix
\section{Proofs of Lemmas~$\ref{lem1}$ and $\ref{lem2}$}

In this appendix, we prove Lemmas~\ref{lem1} and \ref{lem2}.
Denote by $\mathbf{P}$ the probability distribution of $\mathbf{X}$.
First, let us begin with the following lemma.

\begin{lemma}\label{lem0}
	Every finite word $w$ occurs in random data $\mathbf{X}$ with probability $1$.
\end{lemma}
\begin{proof}
	The inclusion relation 
	\[
	\{\text{$w$ never occurs in $\mathbf{X}$}\} \subset \bigcap_{k=0}^\infty \{ \mathbf{X}[k\abs{w}+1,\ldots,(k+1)\abs{w}]\not=w \}
	\]
	holds. This and the independence of $X_1,X_2,\ldots$ imply that 
	\[
	\mathbf{P}\{\text{$w$ never occurs in $\mathbf{X}$}\}
	\le \prod_{k=0}^\infty (1-p(w)) = 0,
	\]
\end{proof}

Lemma~\ref{lem0} implies that 
\begin{equation}
	\sum_{k=0}^\infty \mathbf{P}\{ T_w(\mathbf{X})=k \}
	= \mathbf{P}\biggl( \bigcup_{k=0}^\infty \{ T_w(\mathbf{X})=k \} \biggr)
	= \mathbf{P}\{\text{$w$ occurs in $\mathbf{X}$}\} = 1
	\label{eqA3}
\end{equation}
for every finite word $w$.
Using this fact, we prove Lemmas~\ref{lem1} and \ref{lem2}.

\begin{proof}[Proof of Lemma~$\ref{lem1}$]
	By \eqref{eqA3}, 
	\begin{equation}
		E(w) = \int T_w(\xi)\mathbf{P}(d\xi)
		= \sum_{k=0}^\infty \int_{T_{w_1}=k} T_w(\xi)\mathbf{P}(d\xi).
		\label{eqA1}
	\end{equation}
	When $T_{w_1}(\xi)=k$, setting $\xi'=\xi[k+1,\infty)$, 
	we have $T_w(\xi)=T_w(w_1\xi')+k-\abs{w_1}$.
	Since $T_{w_1}(\xi)=k$ is independent of $\xi'$, it follows that 
	\begin{equation}
	\begin{split}
		\int_{T_{w_1}=k} T_w(\xi)\mathbf{P}(d\xi)
		&= \int_{T_{w_1}=k} \bigl( T_w(w_1\xi')+k-\abs{w_1} \bigr)\mathbf{P}(d\xi)\\
		&= \mathbf{P}\{ T_{w_1}(\mathbf{X})=k \}\int \bigl( T_w(w_1\xi')+k-\abs{w_1} \bigr)\mathbf{P}(d\xi')\\
		&= \mathbf{P}\{ T_{w_1}(\mathbf{X})=k \}\bigl( k + E(w|w_1) \bigr)
	\end{split}\label{eqA2}
	\end{equation}
	for every integer $k\ge0$.
	By \eqref{eqA1}, \eqref{eqA2}, and \eqref{eqA3}, 
	\[
	E(w) = \sum_{k=0}^\infty \mathbf{P}\{ T_{w_1}(\mathbf{X})=k \}\bigl( k + E(w|w_1) \bigr)
	= E(w_1) + E(w|w_1).
	\]
\end{proof}

\begin{proof}[Proof of Lemma~$\ref{lem2}$]
	By Lemma~\ref{lem1}, 
	\begin{align*}
		E(w) &= E(w[1,n-1]) + E(w|w[1,n-1])\\
		&= E(w[1,n-1]) + \int \bigl( T_w(w[1,n-1]\xi)-(n-1) \bigr)\mathbf{P}(d\xi)\\
		&= E(w[1,n-1]) + 1 + \int \bigl( T_w(w[1,n-1]\xi)-n \bigr)\mathbf{P}(d\xi).
	\end{align*}
	Setting $y=\xi(1)$ and $\xi'=\xi[2,\infty)$, we obtain 
	\begin{align*}
		E(w) &= E(w[1,n-1]) + 1 + \sum_{y\in\cX} p(y)\int \bigl( T_w(w[1,n-1]y\xi')-n \bigr)\mathbf{P}(d\xi')\\
		&= E(w[1,n-1]) + 1 + \sum_{y\in\cX} p(y)E(w|w[1,n-1]y).
	\end{align*}
	Since $E(w|w[1,n-1]y)=0$ for $y=w(n)$, 
	the desired equality follows.
\end{proof}

\end{document}